\documentclass{amsart}
\usepackage{amsthm}
\usepackage{amssymb}
\usepackage{graphicx}
\newcommand{\ie}{\text{i.e.,}~}
\newcommand{\lh}{\text{lh}}
\newcommand{\om}{\ensuremath{\omega}}
\newcommand{\set}[2]{\ensuremath{\{#1 \hspace{0.3mm} \mid \hspace{0.3mm} #2\}}}
\newcommand{\ca}[1]{\ensuremath{\mathcal{#1}}}
\newcommand{\iin}{\ensuremath{i \in \omega}}
\newcommand{\R}{\ensuremath{\mathbb R}}
\newcommand\tboldsymbol[1]{%
\protect\raisebox{0pt}[0pt][0pt]{%
$\underset{\widetilde{}}{\boldsymbol{#1}}$}\mbox{\hskip 1pt}}
\newcommand{\n}{\ensuremath{n \in \omega}}
\newcommand{\ep}{\ensuremath{\varepsilon}}
\newcommand{\rfn}[1]{\ensuremath{\{#1\}}}
\newcommand{\bolds}{\ensuremath{\tboldsymbol{\Sigma}}}
\newcommand{\boldp}{\ensuremath{\tboldsymbol{\Pi}}}
\newcommand{\boldd}{\ensuremath{\tboldsymbol{\Delta}}}
\newcommand{\cn}[2]{\ensuremath{#1 \ast #2}}
\newcommand{\scode}{\mathrm{G}}
\newcommand{\fcode}{\mathrm{F}}
\newcommand{\bcode}{\mathrm{B}}
\newcommand{\ocode}{\mathrm{U}}
\newcommand{\bcodefam}{\mathrm{BC}}
\newcommand{\pointcl}{\Gamma}
\newcommand{\gcode}{\ensuremath{{\rm C}_{\pointcl}}}
\newcommand{\baire}{\ensuremath{{\mathcal{N}}}}
\newcommand{\cfo}{{\rm apc}}
\newcommand{\cind}{{\rm apb}}
\newcommand{\sq}[2]{[#1]_{#2}}
\newcommand{\bdbc}{{\rm rbc}}
\newtheorem{theorem}{Theorem}
\newtheorem{lemma}[theorem]{Lemma}
\newtheorem{definition}[theorem]{Definition}
\newtheorem{proposition}[theorem]{Proposition}
\newtheorem{question}[theorem]{Question}
\newtheorem{remark}[theorem]{Remark}

\begin{document}

\title{Uniformity results on the Baire property}

\author[V. Gregoriades]{Vassilios Gregoriades}
\address{Vassilios Gregoriades\\
Via Carlo Alberto, 10\\ 
10123 Turin, Italy}

\email{vassilios.gregoriades@unito.it}

\thanks{Special thanks are owed to {\sc Arno Pauly} for helpful discussions. This article was initiated while the author was a Scientific Associate at TU Darmstadt, Germany. The author is currently a Fellow of the \textit{Programme 2020 researchers : Train to Move}\includegraphics[scale=0.01]{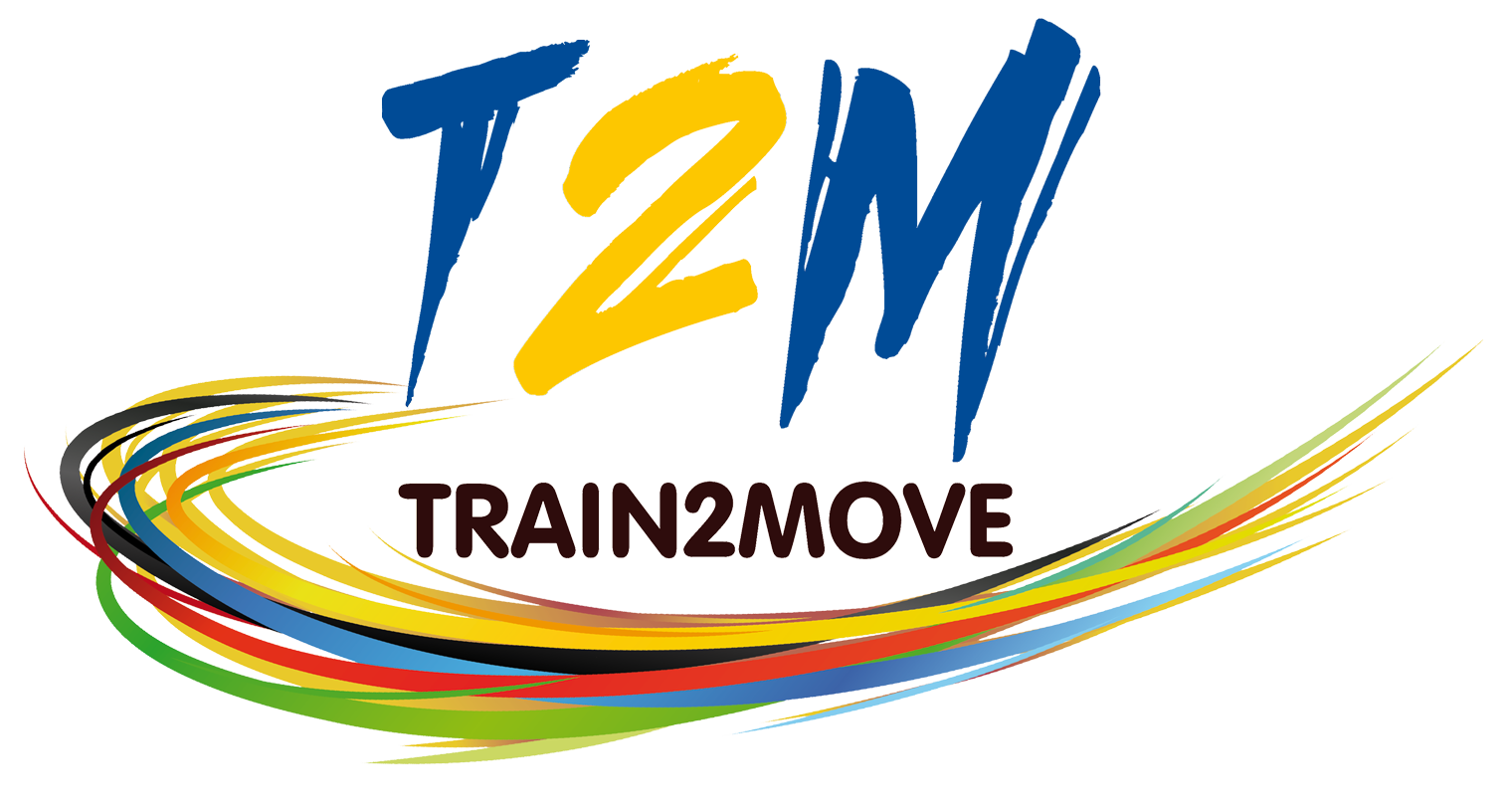} 
at the Mathematics Department ``Guiseppe Peano" of the University of Turin, Italy.}


\date{\today}

\begin{abstract}
We are  concerned with the problem of witnessing the Baire property of the Borel and the projective sets (assuming determinacy) through a sufficiently definable function in the codes. We prove that in the case of projective sets it is possible to satisfy this for almost all codes using a continuous function. We also show that it is impossible to improve this to all codes even if more complex functions in the codes are allowed. We also study the intermediate steps of the Borel hierarchy, and we give an estimation for the complexity of such functions in the codes, which verify the Baire property for actually all codes.
\end{abstract}


\maketitle

\section{Introduction}
We begin with some comments on notation. By $\om$ we mean the first infinite ordinal and by $\om_1$ the first uncountable one. 

We will regularly write $P(x)$ instead of $x \in P$. Our sets will usually be subsets of finite products of spaces. Given $P \subseteq X \times Y$ and $x \in X$ we denote by $P_x$ the \emph{$x$-section} $\set{y \in Y}{P(x,y)}$ of $P$. Many of our results involve functions on the index of a section, \ie we will need to consider sets of the form $P_{u(z)}$ for some function $u$. For reasons of exposition we will denote the section $P_x$ also as $P(x)$. Since the arguments of $P$ will be clear from the context, there will be no danger of confusion by this abuse of notation. 

We denote the set of all finite sequences of naturals (including the empty one) by $\om^{< \om}$. We view the latter set with the discrete topology. A typical element of $\om^{< \om}$ will be denoted by $(u_0,\dots,u_{n-1})$, where $n \in \om$. The case $n = 0$ refers to the empty sequence, which we denote by $\emptyset$. The \emph{length of $u$}, denoted by $\lh(u),$ is the preceding unique $n$.

If $u = (u_0,\dots,u_{n-1})$ and $m < n$ by $u \upharpoonright m$ we mean the \emph{restriction} $(u_0,\dots, u_{m-1})$ of $u$ to $m$.

We fix the function $\langle \cdot \rangle: \om^{< \om} \to \om:$
\[
\langle u_0,\dots, u_{n-1} \rangle = p_0^{u_0+1} \cdot p_1^{u_1+1} \cdot \dots \cdot p_{n-1}^{u_{n-1}+1},
\]
where $n \geq 1$ and $\set{p_i}{i \in \om}$ is the increasing enumeration of all prime numbers. We also set $\langle \emptyset \rangle = 1$. 

Clearly the function $\langle \cdot \rangle$ is injective. Given $s = \langle u_0,\dots,u_{n-1} \rangle$ and $i < n$ we put $(s)_i : = u_i$. If $i \geq n$ or $s$ is not in the image of $\langle \cdot \rangle$ we set $(s)_i = 0$. 
 
We fix once and for all the following enumeration $(q_k)_{k \in \om}$ of the set of all non-negative rational numbers,
\[
q_k = \dfrac{(k)_0}{(k)_1+1}\;.
\]

A \emph{Polish space} is a separable topological space \ca{X}, which is metrizable by a complete metric. For every Polish space \ca{X} we choose a compatible metric $d_{\ca{X}}$ and a dense sequence $(r^{\ca{X}}_i)_{\iin}$. We put
\[
N(\ca{X},s) = \textrm{the $d_{\ca{X}}$-ball with center $r^{\ca{X}}_{(s)_0}$ and radius $q_{(s)_1}$}.
\]
(By zero radius we mean the empty set.)
Therefore the family
\[
\set{N(\ca{X},s)}{s \in \om}
\]
is a basis for the topology of \ca{X}.

Examples of Polish spaces include the \emph{reals} $\R$, the \emph{open unit interval} $(0,1)$ and the \emph{Baire space} $\baire : = \om^\om$ (the set of all infinite sequences of naturals with the product topology). The members of the Baire space are denoted by lowercase Greek letters $\alpha, \beta, \gamma, \dots$.
 
Given $\alpha \in \baire$ we put $\alpha^\ast = (\alpha(1),\alpha(2),\dots)$, \ie 
\[
\alpha^\ast = (i \mapsto \alpha(i+1)).
\]

In the case of the Baire space we make a specific choice for the metric $d_{\baire}$ and the dense sequence $(r^{\baire}_s)_{s \in \om}$, namely
\[
d_\baire(\alpha,\beta) =
\begin{cases}
(n+1)^{-1}, & \ \text{if $\alpha \neq \beta$ and $n$ is the least natural $k$ such that}\\
& \  \alpha(k) \neq \beta(k),\\
0, & \ \text{if $\alpha = \beta$};
\end{cases}
\]
and
\[
r^\baire_s =
\begin{cases}
(u_0,\dots,u_{n-1},0,0,0,\dots), & \ \text{if} \ s = \langle u_0,\dots,u_{n-1}\rangle\\
                                                   & \  \text{for some} \ u = (u_0,\dots,u_{n-1}) \in \om^{<\om},\\
(0,0,0,\dots), & \ \text{else}.
\end{cases}
\]
In other words $r^\baire_s(i) = (s)_i$ for all $s,i$.

With this particular choice of the metric and the dense sequence in the Baire space every set of the form
\[
V_u : = \set{\alpha \in \baire}{(\forall i < \lh(u))[\alpha(i) = u(i)]}, \quad u \in \om^{< \om},
\]
has the form $N(\baire,s_u)$ for some $s_u \in \om$. More specifically
\begin{align*}
V_u 
=& \ \set{\alpha \in \baire}{d_\baire(\alpha,r^{\baire}_{\langle u_0,\dots,u_{\lh(u)-1}\rangle}) < (\lh(u)+1)^{-1}}\\ 
=& \ N(\baire, \langle \langle u_0,\dots,u_{\lh(u)-1}\rangle, \langle 1,\lh(u) \rangle \rangle).
\end{align*}
Thus for the recursive function $\hat{s}: \om^{< \om} \to \om:$
\[
\hat{s}(u) = \langle \langle u_0,\dots,u_{\lh(u)-1}\rangle, \langle 1,\lh(u) \rangle \rangle,
\]
we have $V_u = N(\baire,\hat{s}(u))$ for all $u$.

Conversely for every $s \in \om$ with $N(\baire,s) \neq \emptyset$ there is some $u \in \om^{<\om}$ such that $N(\baire,s) = V_u = N(\baire,\hat{s}(u))$.

To see this we notice that $N(\baire,s) \neq \emptyset$ exactly when $q_{(s)_1} \neq 0$, \ie  $((s)_1)_0 \neq 0$. In the latter case let $n_s$ be the least natural $n$ such that $\dfrac{1}{n+1} < \dfrac{((s)_1)_0}{((s)_1)_1+1}\;$. Then $N(\baire,s)$ consists exactly of all $\alpha \in \baire$ which agree with $r^\baire_{(s)_0}$ on the first $n_s-1$ coordinates, \ie
\[
N(\baire,s) = \set{\alpha}{(\forall i < n_s)[\alpha(i) = ((s)_0)_i]}.
\]
(If $n_s = 0$ then $N(\baire,s) = \baire$.) 

Hence $N(\baire,s) = V_u$, where $u = (((s)_0)_0,\dots,((s)_0)_{n_s-1})$. We also see  that one can in fact choose $u$ as a recursive function on $s$, \ie the recursive  function $\check{u}: \om \to \om^{<\om}$:
\[
\check{u}(s)=
\begin{cases}
((s)_0)_0,\dots,((s)_0)_{n_s-1}), & \ \text{if $((s)_1)_0 \neq 0$ and $n_s$ is the least natural $n$}\\
& \  \text{such that $\dfrac{1}{n+1} < \dfrac{((s)_1)_0}{((s)_1)_1+1}\;$};\\
0, & \ \text{else};
\end{cases}
\]
satisfies $N(\baire,s) = V_{\check{u}(s)}$ for all $s$ for which $N(\baire,s) \neq \emptyset$.

\subsubsection*{Measurability of functions} Suppose that $\pointcl$ is a class, also called \emph{pointclass}, of sets in Polish spaces. The family of all subsets of a Polish space \ca{X} which belong to the pointclass $\pointcl$ is denoted by $\pointcl \upharpoonright \ca{X}$.

A function $f: \ca{X} \to \ca{Y}$ between Polish spaces is \emph{$\pointcl$-measurable} if for all open $V \subseteq \ca{Y}$ the prei-image $f^{-1}[V]$ is in $\pointcl \upharpoonright \ca{X}$.  We will need to consider the measurability of functions, which are defined on arbitrary subsets of Polish spaces. If $\ca{X}$ and $\ca{Y}$ are Polish spaces and $A$ is a non-empty subset of \ca{X} we say that a function $f: A \to \ca{Y}$ is \emph{$\pointcl$-measurable} if for all open $V \subseteq \ca{Y}$ there is some $W \in \pointcl \upharpoonright \ca{X}$ such that $f^{-1}[V] = W \cap A$.
 
\subsubsection*{Universal systems} Given a Polish space \ca{X} and a set $G \subseteq \baire \times \ca{X}$ we say that $G$ \emph{parametrizes $\pointcl \upharpoonright \ca{X}$} if for all $P \subseteq \ca{X}$ we have that $P$ is in $\pointcl$ exactly when for some $\alpha \in \baire$ it holds $P = G_\alpha$. The set $G$ is \emph{universal for $\pointcl \upharpoonright \ca{X}$} if $G \in \pointcl$ and parametrizes $\pointcl \upharpoonright \ca{X}$.

A class $(G^\ca{X}_{\pointcl})_{\ca{X}\text{: Polish}}\;$, is a parametrization (universal) system for $\pointcl$ if for each $\ca{X}$ the set $G^\ca{X}_{\pointcl}$ parametrizes (resp. is universal for) $\pointcl\upharpoonright \ca{X}$. 

In the sequel we employ the usual pointclasses $\bolds^i_n$, $\boldp^i_n$, $\boldd^i_n$, $i=0,1, \n$, from descriptive set theory. E.g., $\bolds^0_1$ consists of all open sets, $\boldp^0_1$ of all closed sets, $\bolds^0_2$ of all $F_\sigma$ sets and so on. The members of $\bolds^1_1$ are the \emph{analytic} sets. A set is \emph{co-analytic} if its complement is analytic.

The pointclasses $\bolds^i_n$, and  consequently $\boldp^i_n$, $i=0,1, \n$, admit a universal system in a natural way:

For every Polish space \ca{X} we define the set $\ocode^\ca{X}, \fcode^\ca{X} \subseteq \baire \times \ca{X}$ by
\begin{align*}
\ocode^\ca{X}(\alpha,x) \iff& (\exists n)[x \in N(\ca{X},\alpha(n))]\\
\fcode^\ca{X}(\alpha,x) \iff& \neg \ocode^\ca{X}(\alpha,x).
\end{align*}
It is clear that the sets $\ocode^\ca{X}, \fcode^\ca{X}$ are universal for $\tboldsymbol{\Sigma}^0_1 \upharpoonright \ca{X}$ and $\tboldsymbol{\Pi}^0_1 \upharpoonright \ca{X}$ respectively. This construction proceeds of course to all levels of the Borel hierarchy, but for the moment we will skip them and deal with the projective sets. By recursion we define $\scode^\ca{X}_n \subseteq \baire \times \ca{X}$ as follows
\begin{align*}
\scode^\ca{X}_1(\alpha,x) \iff& (\exists \gamma) \fcode^{\ca{X} \times \baire}(\alpha,x,\gamma)\\
\scode^\ca{X}_{n+1}(\alpha,x) \iff& (\exists \gamma) \neg \scode^{\ca{X} \times \baire}_n(\alpha,x,\gamma).
\end{align*}
An easy inductive argument shows that $\scode_n^\ca{X}$ is universal for $\bolds^1_n \upharpoonright \ca{X}$ for all $n \geq 1$. By saying that $\alpha$ is for instance an \emph{open code} for $P\subseteq \ca{X}$ we mean that $P$ is the $\alpha$-section of the set $\ocode^\ca{X}$. We fix the preceding universal sets throughout the rest of this article.

One key aspect of our chosen universal system systems is that they are \emph{good}, \ie for every space \ca{X} of the form $\om^n \times \baire^m$, where $n,m \geq 0$, and every Polish space $\ca{Y}$ there is a continuous function $S: \baire \times \ca{X} \to \baire$ such that for all $\ep$, $\alpha$, $x$ it holds
\[
\gcode^{\baire \times \ca{X}}(\ep,\alpha,x) \iff \gcode^{\ca{X}}(S(\ep,\alpha),x),
\]
where $\pointcl$ is the pointclass under discussion and $(\gcode^\ca{X})_{\ca{X}}$ is the corresponding universal system. This is verified by straight-forward computations, which we omit.

\subsubsection*{Effective notions}

We will assume some familiarity with the topic of effective descriptive theory. The usual reference to the latter is \cite{yiannis_dst}. Let us recall some fundamental notions. A sequence $(x_n)_{\n}$ in the complete separable metric space $(\ca{X},d_{\ca{X}})$ is a \emph{recursive presentation} of $(\ca{X},d_{\ca{X}})$ if it forms a dense subset of $(\ca{X},d_{\ca{X}})$ and the relations $P, Q \subseteq \om^3$, defined by
\begin{align*}
P(i,j,k) \iff& d(x_i,x_j) < q_k\\
Q(i,j,k) \iff& d(x_i,x_j) \leq q_k
\end{align*}
are recursive. The metric space $(\ca{X},d_{\ca{X}})$ is \emph{recursively presented} if it admits a recursive presentation. We may assume without loss of generality that the sequence $(r^{\ca{X}}_i)_{\iin}$ that we fixed above, is the recursive presentation of $(\ca{X},d_{\ca{X}})$, when the latter is recursively presented. In fact the sequence $(r^\baire_s)_{s \in \om}$ is a recursive presentation of  $(\baire, d_{\baire})$. Other examples of recursively presented metric spaces are the reals and $\om$ with the usual metric, and $\om^{< \om}$ with the discrete metric.

The class of \emph{effectively open} or else \emph{semirecursive} sets, denoted by $\Sigma^0_1$ is the class of all sets of the form $\ocode^\ca{X}_\alpha$ for some recursive $\alpha: \om \to \om$, where $(\ca{X},d_{\ca{X}})$ is recursively presented. By replacing the term ``recursive" with ``$\ep$-recursive" one defines the class $\Sigma^0_1(\ep)$. For each of the classical pointclasses $\bolds^i_n$, $\boldp^i_n$ $\boldd^i_n$ there are the respective effective pointclasses $\Sigma^i_n(\ep)$, $\Pi^i_n(\ep)$, $\Delta^i_n(\ep)$. As it is well-known the $\Sigma^i_n(\ep)$ subsets of recursively presented metric spaces are exactly the $\ep$-recursive sections of the sets $\scode^\ca{X}_n$, and similarly for the other pointclasses.

A function $f: \ca{X} \to \ca{Y}$ between recursively presented metric spaces is $\Gamma$-recursive (where $\Gamma$ is one of the preceding effective pointclasses) if the relation $R^f \subseteq \ca{X} \times \om$ defined by
\[
R^f(x,s) \iff f(x) \in N(\ca{Y},s)
\]
is in $\Gamma$. We say ``recursive" instead of ``$\Sigma^0_1$- recursive. The notion of $\Gamma$-recursiveness can be considered as the effective analogue of $\tboldsymbol{\Gamma}$-measurability, e.g., a recursive function is the effective analogue of continuity.

\subsection*{Borel codes}

The Borel classes $\bolds^0_n$, $n \geq 1$ are naturally extended to the classes $\bolds^0_\xi$ of transfinite order, where $1 \leq \xi < \om_1$. The latter classes admit universal sets but here we will use a different type of encoding introduced by Louveau-Moschovakis \cite{louveau_a_separation_theorem_for_sigma_sets} and \cite{yiannis_dst}.

A partial function $f: X \rightharpoonup Y$ from $X$ to $Y$ is a $Y$-valued function $f$, which is defined on a (perhaps empty) subset of $X$. We write $f(x) \downarrow$ to denote that $x$ is in the domain of $f$.

In the sequel we denote by $\rfn{\alpha}^{\om, \baire}$ or simpler by $\rfn{\alpha}$ the largest partial function $f: \om \rightharpoonup \baire$, which is computed on its domain by $\ocode^{\om \times \om}_\alpha$, \ie $\rfn{\alpha}^{\om, \baire} \downarrow$ exactly when there is a unique $\beta \in \baire$ such that for all $s \in \om$ we have that
\[
\beta \in N(\baire,s) \iff \ocode^{\om \times \om}(\alpha,n,s).
\]
In this case we let $\rfn{\alpha}^{\om, \baire}$ be the unique $\beta$ as above, so that when $\rfn{\alpha}(n) \downarrow$ we have
\begin{align}\label{equation compute partial function}
\rfn{\alpha}(n) \in N(\baire,s) \iff \ocode^{\om \times \om}(\alpha,n,s)
\end{align}
for all $s \in \om$.

We define by recursion the families $\bcodefam_\xi \subseteq \baire$, $\xi < \om_1$ as follows
\begin{align*}
\alpha \in \bcodefam_0 \iff& \alpha(0) = 0,\\
\alpha \in \bcodefam_\xi \iff& \alpha(0)=1 \ \& \ (\forall n)(\exists \zeta < \xi)[\rfn{\alpha^\ast}(n) \downarrow \ \& \ \rfn{\alpha^\ast}(n) \in \bcodefam_\zeta].
\end{align*}
The set of \emph{Borel codes} is 
\[
\bcodefam = \cup_{\xi < \om_1} \bcodefam_\xi.
\]
For $\alpha \in \bcodefam$ we put
\[
|\alpha| = \textrm{the least $\xi$ such that $\alpha \in \bcodefam_\xi$}.
\]
Given a Polish space $\ca{X}$ we define $\pi^{\ca{X}}_\xi : \bcodefam_\xi \to \bolds^0_\xi \upharpoonright \ca{X}$ by recursion on $\xi < \om_1$,
\begin{align*}
\pi^\ca{X}_0(\alpha) =& \ \ca{X} \setminus N(\ca{X},\alpha(1))\\
\pi^\ca{X}_\xi(\alpha) =& \ \cup_{n} \ca{X} \setminus \pi^\ca{X}_{|\rfn{\alpha^\ast}(n)|}(\rfn{\alpha^\ast}(n)),
\end{align*}
whereas by $\bolds^0_0 \upharpoonright \ca{X}$ we mean the family $\set{\ca{X} \setminus N(\ca{X},s)}{s \in \om}$. (It is clear that the members of $\bolds^0_1$ are the countable unions of the complements of  $\bolds^0_0$ sets.)

One can verify that the preceding functions $\pi^\ca{X}_\xi$ are surjective, and so every set in $\bolds^0_\xi \upharpoonright \ca{X}$ has a code through $\pi^\ca{X}_\xi$. It is also not hard to see that $\bcodefam_\zeta \subseteq \bcodefam_\xi$ and $\pi^\ca{X}_\xi \upharpoonright \bcodefam_\zeta = \pi^\ca{X}_\zeta$ for all $1 \leq \zeta \leq \xi$.

It follows that the function
\begin{align*}
\pi^\ca{X}: \bcodefam \to {\rm Borel}(\ca{X}): \pi^\ca{X}(\alpha) =& \ \pi^\ca{X}_{|\alpha|}(\alpha)\\
                                                                                 =& \ \pi^\ca{X}_{\xi}(\alpha) \quad (\textrm{when} \ \xi > |\alpha| \geq 1),
\end{align*}
gives a parametrization of all Borel subsets of \ca{X}.

Finally we consider the sets $\bcode^\ca{X}_\xi \subseteq \baire \times \ca{X}$,
\[
\bcode^\ca{X}_\xi(\alpha,x) \iff \alpha \in \bcodefam_\xi \ \& \ x \in \pi^\ca{X}(\alpha).
\]
From the preceding comments it is clear that $\bcode^\ca{X}_\xi$ parametrizes $\bolds^0_\xi \upharpoonright \ca{X}$.

\begin{remark}\normalfont
\label{remark equivalent coding of Borel sets}
One can encode the Borel sets in the same manner as above by replacing $\rfn{\alpha}(n) \in \baire$ with $(\alpha)_n \in \baire$, where
\[
(\alpha)_n(t) = \alpha(\langle n,t \rangle) \quad t,n \in \om.
\]
This encodes the Borel sets in a slightly different way, e.g., the corresponding set of Borel codes is not the same as the $\bcodefam$ from above, although it has similar properties. Moreover these two ways of encoding the Borel sets are equivalent, in the sense that we can pass from a code with respect to one way to a code with respect to the other way using continuous (in fact recursive) functions, see \cite[7B.8]{yiannis_dst}.

While $(\alpha)_n$ is easier than $\rfn{\alpha}(n)$ to understand, there are certain technical advantages using the latter, for example when applying the Kleene Recursion Theorem. In our proofs though, it is not necessary to use the $\rfn{\alpha}(n)$-coding, so the reader may very well interpret $\rfn{\alpha}(n)$ as $(\alpha)_n$ in all of our subsequent arguments (with the exception of Remark \ref{remark the set Def is pi02} which refers specifically to the functions $\rfn{\alpha}$). 

Even though the advantages of the $\rfn{\alpha}(n)$-coding do not manifest in this article, we nevertheless opt for using it, as we think that this way the proofs provide more information to the research community.
\end{remark}

\section{The Results}

Recall that  a subset $P$ of a Polish space $\ca{X}$ has the \emph{Baire property} if there exists an open set $U \subseteq \ca{X}$ such that the symmetric difference $P \triangle U : = (P \setminus U) \cup (U \setminus P)$ is meager.

A fundamental property of the analytic sets is that they have the Baire property (Lusin-Sierpinski \cite{lusin_sierpinski_sur_un_enseble_non_measurable_B}). Moreover under the axiom of \emph{determinacy} of $\pointcl$-games on the naturals, every subset of a Polish space, which is in $\pointcl$ has the Baire property (Banach-Mazur, Oxtoby, see \cite[8.35]{kechris_classical_dst} and also \cite[6A.16]{yiannis_dst}).

The question of witnessing a given property in a uniform continuous way is prominent in descriptive set theory and related areas. For example this uniform approach has applications in the decomposability of Borel-measurable functions and a still open conjecture on the extension of the Jayne-Rogers Theorem, see \cite{kihara_decomposing_Borel_functions_using_the_Shore_Slaman_join_theorem,gregoriades_kihara_ng_Turing_degrees_in_Polish_spaces_and_decomposability_of_Borel_functions}. Another application is the Suslin-Kleene Theorem, which extends the fundamental fact of recursion theory that  ${\sl HYP} = \Delta^1_1$.

In this note we focus on the question of witnessing the Baire property in a continuous way in the codes, \ie given a pointclass $\pointcl$, which admits a universal system $(\gcode^{\ca{X}})_{\ca{X}}$ and for which every set $P \in \pointcl \upharpoonright \ca{X}$ (with \ca{X} being Polish) has the Baire property, does there exist a sufficiently definable function $u: \baire \to \baire$ such that the symmetric difference
\[
\gcode^{\ca{X}}(\alpha) \triangle \ocode^{\ca{X}}(u(\alpha))
\]
is a meager set?

\subsection*{Almost uniform transition} Classes of sets, which have sufficiently good structure and have the Baire property, answer the preceding question almost affirmatively. 

\begin{lemma}
\label{lemma gamma has the Baire property}
Assume that $\pointcl$ admits a universal system $(G^ \ca{Y}_{\pointcl})_{\ca{Y}}$ and that every subset of a Polish space which is in $\pointcl$ has the Baire property. Then for every Polish space \ca{X} there exists a continuous function $u^\ca{X}_{\pointcl}: \baire \to \baire$
such that for almost all $\alpha \in \baire$ the set 
\[
G^\ca{X}_{\pointcl}(\alpha) \triangle \ocode^\ca{X}(u^\ca{X}_{\pointcl}(\alpha))
\]
is meager.
\end{lemma}

\begin{proof}Since $G^\ca{X}_{\pointcl}$ belongs in $\pointcl$ it has in particular the Baire property. Hence there is some open set $V \subseteq \baire \times \ca{X}$ such that the symmetric difference $H: = G^\ca{X}_{\pointcl} \triangle V$ is meager. Since $H$ is meager it follows from the Kuratowski-Ulam Theorem that for almost all $\alpha \in \baire$ the section $H(\alpha)$ is meager.

Since $V$ is open there is some $\ep \in \baire$ such that $V = \ocode^{\baire \times \ca{X}}(\ep)$. We thus have
\[
V(\alpha,x) \iff \ocode^{\baire \times \ca{X}}(\ep,\alpha,x) \iff \ocode^\ca{X}(S(\ep,\alpha),x)
\]
for some suitable continuous function $S: \baire \times \baire \to \baire$. Hence $V(\alpha) = \ocode^{\ca{X}}(\alpha)$ for all $\alpha \in \baire$. We take $u^\ca{X}_{\pointcl}(\alpha) = S(\ep,\alpha)$, $\alpha \in \baire$.

Finally $H(\alpha) = G^\ca{X}_{\pointcl}(\alpha) \triangle V(\alpha) = G^\ca{X}_{\pointcl}(\alpha) \triangle \ocode^\ca{X}(u^\ca{X}_{\pointcl}(\alpha))$ for all $\alpha \in \baire$, hence the set $G^\ca{X}_{\pointcl}(\alpha) \triangle \ocode^\ca{X}(u^\ca{X}_{\pointcl}(\alpha))$ is meager for almost all $\alpha \in \baire$.
\end{proof}

We thus we obtain

\begin{proposition}[Axiom of $\bolds^1_n$ Determinacy for $n >1$]
\label{proposition continuous transition analytic}
For every Polish space \ca{X} and every $\n$ with $n \geq 1$ there exists a continuous function $u^ \ca{X}_n : \baire \to \baire$ such that for almost all $\alpha \in \baire$ the set 
\[
\scode^\ca{X}_{n}(\alpha) \triangle \ocode^\ca{X}(u^ \ca{X}_n(\alpha))
\]
is meager.
\end{proposition}

We will show that the preceding results are optimal (Theorem \ref{theorem optimal main}), but before we do this we will be concerned with

\subsection*{The intermediate steps of the Borel-hierarchy} Here we provide the analogous result at the $\xi$-level of the Borel-hierarchy. In these cases one can actually obtain the transition \emph{for all} codes, instead for almost all. On the other hand the complexity of the uniformity functions does not seem to be (and in fact due to Theorem \ref{theorem optimal main} is not) bounded.

\begin{definition}\normalfont
\label{definition of upper bound}
We define the functions $\bdbc, \cind: \om_1 \setminus \{0\} \to \om_1\setminus \{0\}$ by
\begin{align*}
\bdbc(\xi) =& \ \text{the least $\eta < \om_1$ such that} \ \bcodefam_\xi \ \in \boldp^0_\eta,\\
\cind(1) =& \ 1,\\
\cind(\xi) = & \ \sup\set{\max\{\cind(\eta), \bdbc(\eta)+1\}}{1 \leq \eta < \xi} + 1.
\end{align*}
\end{definition}

Here ``$\bdbc$" stands for ``rank of Borel coding" and ``$\cind$" for ``approximation to Borel".  In Proposition \ref{proposition upper bound for bdbc and cind} we give an estimation for the values of the preceding functions.

We recall the partial functions $\rfn{\alpha}^{\om,\baire} \equiv \rfn{\alpha}: \om \rightharpoonup \baire$ given in the Introduction. The following remark summarizes some basic properties of these functions, which are more or less folklore in the area. The proof involves a somewhat long list of computations but nevertheless we put it down for reasons of self-containedness.

\begin{remark}\normalfont
\label{remark the set Def is pi02}
The set
\[
Dom(\alpha,n) \iff \rfn{\alpha}(n) \downarrow
\]
is $\Pi^0_2$ and consequently it is $G_\delta$.\smallskip

Moreover the function
\[
eval: Dom \times \om^2 \to \om: (\alpha,n,k) \mapsto \rfn{\alpha}(n)(k)
\]
is recursive on $Dom \times \om^2$, \ie there is a semirecursive set $R^{eval} \subseteq Dom \times \om^3$ such that for all $(\alpha,n) \in Dom$ and all $k,s \in \om$ we have that
\[
eval(\alpha,n,k) = s \iff R^{eval}(\alpha,n,k,s).
\]
In particular the partial function
\[
\alpha \mapsto \rfn{\alpha} \in {\baire}^\om
\]
is well-defined on $\set{\beta \in \baire}{(\forall n)Def(\beta,n)}$ and continuous on the latter set.
\end{remark}

\begin{proof}
We consider first the recursive functions $\hat{s}: \om^{< \om} \to \om$ and $\check{u}: \om^{<\om} \to \om$ given in the Introduction for which have that
\[
N(\baire,\hat{s}(u)) = V_u = \set{\alpha \in \baire}{(\forall i < \lh(u))[\alpha(i) = u(i)]}
\]
for all $u$, and 
\[
N(\baire,s) = V_{\check{u}(s)}
\]
for all $s$ with $N(\baire,s) \neq \emptyset$. 

Below we denote by $\cn{u}{v}$ the \emph{concatenation} of $u \in \om^{<\om}$ with $v \in \om^{<\om}$, \ie \newline $\cn{u}{v} = (u_0,\dots,u_{\lh(u)-1},v(0),\dots,v_{\lh(v)-1})$.\smallskip

\emph{Claim.} For all $\alpha,n$ we have $\rfn{\alpha}(n) \downarrow$ exactly when 
\begin{align}
&  \label{equation remark the set Def is pi02 A} \ \ocode^{\om \times \om}(\alpha,n,\hat{s}(\emptyset)) \ \&\\ \nonumber
& \  (\forall u \in \om^{< \om})[\ocode^{\om \times \om}(\alpha,n,\hat{s}(u)) \ \longrightarrow \ (\exists i)\ocode^{\om \times \om}(\alpha,n,\hat{s}({\cn{u}{(i)}}))] \ \&\\ \nonumber
& \  (\forall u \in \om^{< \om})(\forall i, j)[\left(\ocode^{\om \times \om}(\alpha,n,\hat{s}({\cn{u}{(i)}})) \ \& \ \ocode^{\om \times \om}(\alpha,n,\hat{s}({\cn{u}{(j)}}))\right) \ \longrightarrow \ i = j] \ \&\\ \nonumber
& \  (\forall u \in \om^{<\om})(\forall i \in \om)[\ocode^{\om \times \om}(\alpha,n,\hat{s}(\cn{u}{(i)})) \ \longrightarrow \ \ocode^{\om \times \om}(\alpha,n,\hat{s}(u))] \ \&\\ \nonumber
& \ (\forall s)[\ocode^{\om \times \om}(\alpha,n,s) \ \longrightarrow \ ((s)_1)_0 \neq 0] \ \&\\ \nonumber
& \ (\forall s \in \om)[((s)_1)_0 \neq 0 \ \& \ \ocode^{\om \times \om}(\alpha,n,\hat{s}(\check{u}(s))) \ \longrightarrow \ \ocode^{\om \times \om}(\alpha,n,s)] \ \&\\ \nonumber
& \ (\forall s \in \om)[((s)_1)_0 \neq 0 \ \& \ \ocode^{\om \times \om}(\alpha,n,s) \ \longrightarrow \ \ocode^{\om \times \om}(\alpha,n,\hat{s}(\check{u}(s)))].
\end{align}
The first three conjuncts of (\ref{equation remark the set Def is pi02 A}) give us a way to determine the value of $\rfn{\alpha}(n)$ step by step, by picking each time the unique $i$ for which $\ocode^{\om \times \om}(\alpha,n,\hat{s}({\cn{u}{(i)}}))$ holds. The next one ensures that we can pass to smaller neighborhoods and the fifth conjunct says that $N(\baire,s)$ is non-empty whenever $\ocode^{\om \times \om}(\alpha,n,s)$ holds. The sixth conjunct allows us to pass from $\ocode^{\om \times \om}(\alpha,n,\hat{s}(u))$ to $\ocode^{\om \times \om}(\alpha,n,s)$ for some specific (but sufficiently many) choices of $u$, whenever $N(\baire,s) \neq \emptyset$ equivalently whenever $((s)_1)_0 \neq 0$. Finally the last conjunct is the converse transition. (The necessity for these conditions will become clear below.)

It is not hard to verify that condition (\ref{equation remark the set Def is pi02 A}) defines a $\Pi^0_2$ subset of $\baire  \times \om$, so if we prove the claim we are done with the first assertion.\smallskip

\emph{Proof of the claim.} Assume first that $\rfn{\alpha}(n)$ is defined, so that there exists a unique $\beta \in \baire$ such that for all $s \in \om$, $\beta \in N(\baire,s) \iff \ocode^{\om \times \om}(\alpha,n,s)$.

Since $\beta \in \baire = N(\baire,\hat{s}(\emptyset))$ we have $\ocode^{\om \times \om}(\alpha,n,\hat{s}(\emptyset))$. Assume that $\ocode^{\om \times \om}(\alpha,n,\hat{s}(u))$ holds for some $u \in \om^{< \om}$. Then $\beta \in N(\baire,\hat{s}(u))$, \ie $u$ is an initial segment of $\beta$. We take $i = \beta(\lh(u))$ so that $\cn{u}{(i)}$ is also an initial segment of $\beta$, which means that $\beta \in N(\baire,\hat{s}(\cn{u}{(i)}))$. Thus $\ocode^{\om \times \om}(\alpha,n,\hat{s}({\cn{u}{(i)}}))$. Now assume that both conditions $\ocode^{\om \times \om}(\alpha,n,\hat{s}({\cn{u}{(i)}}))$ and $\ocode^{\om \times \om}(\alpha,n,\hat{s}({\cn{u}{(j)}}))$ hold for some $u, i, j$. Then $\cn{u}{(i)}$ and $\cn{u}{(j)}$ are initial segments of $\beta$ and so $i = \beta(\lh(u)) = j$. 

For the next conjunct of (\ref{equation remark the set Def is pi02 A})  assume that $\ocode^{\om \times \om}(\alpha,n,\hat{s}(\cn{u}{(i)})$ holds. Then from our hypothesis $\beta \in N(\baire,\hat{s}(\cn{u}{(i)}) = V_{\cn{u}{(i)}} \subseteq V_{u} = N(\baire,\hat{s}(u))$. Hence $\ocode^{\om \times \om}(\alpha,n,\hat{s}(u)$ holds as well.

Moreover it is clear that $\ocode^{\om \times \om}(\alpha,n,s)$ implies that $N(\baire,s) \neq \emptyset$ as it contains $\beta$.

Finally assume that $s \in \om$ satisfies $((s)_1)_0 \neq 0$, \ie $N(\baire,s) \neq \emptyset$. Then $N(\baire,\hat{s}(\check{u}(s))) = V_{\check{u}(s)} = N(\baire,s)$. So
\[
\ocode^{\om \times \om}(\alpha,n,\hat{s}(\check{u}(s))) \iff \beta \in N(\baire,\hat{s}(\check{u}(s))) \iff N(\baire,s) \iff \ocode^{\om \times \om}(\alpha,n,s).
\]
This proves the last two conjuncts of condition (\ref{equation remark the set Def is pi02 A})  and the left-to-right-hand implication of our claim is settled.

For the converse direction, we define by recursion
\[
\beta(k) = \ \text{the unique $i \in \om$ such that} \ \ocode^{\om \times \om}(\alpha,n,\hat{s}(\cn{(\beta(0),\dots,\beta(k-1))}{(i)})),
\]
where $k \in \om$. From our hypothesis $\beta$ is well-defined and $\ocode^{\om \times \om}(\alpha,n,\hat{s}(u))$ for all $u$, which are initial segments of $\beta$.

We need to show that $\beta \in N(\baire,s) \iff \ocode^{\om \times \om}(\alpha,n,s)$ for all $s$, and that $\beta$ is the unique such point in the Baire space.

Assume first that $\beta$ is a member of $N(\baire,s)$. In particular the latter set is not empty and so from our preceding remarks we have that $N(\baire,s) = V_{\check{u}(s)} = N(\baire,\hat{s}(\check{u}(s)))$. Thus $\check{u}(s)$ is an initial segment of $\beta$ and hence $\ocode^{\om \times \om}(\alpha,n,\hat{s}(\check{u}(s)))$ holds. From the sixth conjunct of condition (\ref{equation remark the set Def is pi02 A}) it follows that $\ocode^{\om \times \om}(\alpha,n,s)$ holds as well. Conversely if $\ocode^{\om \times \om}(\alpha,n,s)$ holds then from the fifth conjunct $N(\baire,s) \neq \emptyset$ and so $N(\baire,s) = V_{\check{u}(s)}$. We show by induction on $k < \lh(\check{u}(s))$ that $\check{u}(s)(i) = \beta(i)$ for all $i \leq k$. Let $k < \lh(\check{u}(s))$ and assume that $\beta(i)= \check{u}(s)(i)$ for all $i < k$. From the last conjunct of (\ref{equation remark the set Def is pi02 A})  we have $\ocode^{\om \times \om}(\alpha,n,\hat{s}(\check{u}(s)))$ and from the fourth one it follows that $\ocode^{\om \times \om}(\alpha,n,\hat{s}(\check{u}(s) \upharpoonright (k+1)))$, \ie $\ocode^{\om \times \om}(\alpha,n,\hat{s}(\beta(0),\dots,\beta(k-1),\check{u}(s)(k)))$. Since by the definition of $\beta$ we have $\ocode^{\om \times \om}(\alpha,n,\hat{s}(\beta(0),\dots,\beta(k-1),\beta(k)))$ it follows from the third conjunct of (\ref{equation remark the set Def is pi02 A}) that $\check{u}(s)(k) = \beta(k)$. This completes the inductive step and so $\check{u}(s)$ is an initial segment of $\beta$. Thus $\beta \in V_{\check{u}(s)} = N(\baire,s)$.

Therefore we have shown that $\beta \in N(\baire,s) \iff \ocode^{\om \times \om}(\alpha,n,s)$ for all $s$. To show the uniqueness assume that $\beta'$ satisfies the latter equivalence. As before we show by induction that $\beta(k) = \beta'(k)$ for all $k$. Suppose that $\beta (i) = \beta' (i)$ for all $i < k$. Evidently 
\[
\beta' \in N(\baire,\hat{s}((\beta'(0),\dots,\beta'(k-1),\beta'(k))) = N(\baire,\hat{s}((\beta(0),\dots,\beta(k-1),\beta'(k))).
\]
By our assumption about $\beta'$ it follows $\ocode^{\om \times \om}(\alpha,n,\hat{s}((\beta(0),\dots,\beta(k-1),\beta'(k))))$. Applying the same argument to $\beta$ yields $\ocode^{\om \times \om}(\alpha,n,\hat{s}((\beta(0),\dots,\beta(k-1),\beta(k))))$. It follows from the third conjunct of (\ref{equation remark the set Def is pi02 A}) that $\beta(k) = \beta'(k)$.

This finishes the proof of the claim.\smallskip

Regarding the second assertion, we have for all $(\alpha,n) \in Dom$ and all $k,s \in \om$,
\begin{align*}
eval(\alpha,n,k) = s 
\iff& \ \rfn{\alpha}(n)(k) = s\\
\iff& \ (\exists u \in \om^{< \om})[\lh(u) = k \ \& \ u_{k-1} = s \ \& \ \rfn{\alpha}(n) \in N(\baire, \hat{s}(u))]\\
\iff& \ (\exists u \in \om^{< \om})[\lh(u) = k \ \& \ u_{k-1} = s \ \& \ \ocode^{\om \times \om}(\alpha,n, \hat{s}(u)))]
\end{align*}
where in the last equivalence we use equality (\ref{equation compute partial function}) from the Introduction.

The latter condition defines a $\Sigma^0_1$ subset of $\baire \times \om^3$, from which it follows easily that the partial function $eval$ is recursive on its domain.

The last assertion follows easily from the fact that semirecursive sets are open.
\end{proof}

\begin{proposition}
\label{proposition upper bound for bdbc and cind}
For all countable $\xi \geq 1$ we have that 
\[
\bdbc(\xi) \leq \xi+1 \quad \text{and} \quad \cind(\xi) \leq \xi+2.
\]
\end{proposition}

\begin{proof}
By induction on $\xi$. It is evident that $\cind(1) \leq 3$. Also from Remark \ref{remark the set Def is pi02} and the continuity of the function $(\alpha \mapsto \alpha^\ast)$ it follows that the set $\bcodefam_1$ is $\boldp^0_2$ and hence $\bdbc(1) = 2$.

Suppose that $1 < \xi < \om_1$ and assume that the inequalities are true for all $1 \leq \eta < \xi$. We first assume that $\xi$ is a limit ordinal and we fix an enumeration $(\xi_m)_{m \in \om}$ of $\xi$. 

Clearly $\bcodefam_\xi = \cap_n\cup_m A_{m,n}$, where 
\[
A_{m,n} = \set{\alpha \in \baire}{\alpha(0) = 1 \ \& \ \rfn{\alpha^\ast}(n) \downarrow \ \& \ \rfn{\alpha^\ast}(n) \in \bcodefam_{\xi_m}}.
\]
From our inductive hypothesis we  have that $\bdbc(\xi_m) \leq \xi_m+1$ for all $m$, and so from Remark \ref{remark the set Def is pi02} each $A_{m,n}$ is a $\boldp^0_{\max\{\xi_m+1,2\}}$ set. Since $\xi$ is limit with $\xi > \xi_m$ for all $m$, it follows that $\xi > \max\{\xi_m+1,2\}$ for all $m$, and so the union $\cup_m A_{m,n}$ is $\bolds^0_\xi$ for all \n. Hence $\bcodefam_\xi$ is $\boldp^0_{\xi+1}$ and so $\bdbc(\xi) \leq \xi+1$.

It also follows from our inductive hypothesis that
$
\max\{\cind(\eta), \bdbc(\eta)+1\} \leq \eta+2
$
for all $\eta < \xi$ and using again that $\xi$ is limit, 
\[
\cind(\xi) \leq \sup\set{\eta+2}{1 \leq \eta < \xi} + 1 = \xi+1 < \xi +2.
\]
If $\xi$ is a successor, say $\xi = \eta_0+1$, with $\eta_0 \geq 1$ then
\[
\bcodefam_\xi = \cap_n \set{\alpha \in \baire}{\alpha(0) = 1 \ \& \ \rfn{\alpha^\ast}(n) \downarrow \ \& \ \rfn{\alpha^\ast}(n) \in \bcodefam_{\eta_0}}.
\]
It follows from our inductive hypothesis and Remark \ref{remark the set Def is pi02} that $\bcodefam_\xi$ is the countable intersection of $\boldp^0_{\eta_0+1}$ sets. Therefore $\bcodefam_\xi$ is in $\boldp^0_{\eta_0+1} = \boldp^0_\xi \subseteq \boldp^0_{\xi+1}$, \ie $\bdbc(\xi) \leq \xi+1$. Moreover from our inductive hypothesis we have
$
\max\{\cind(\eta), \bdbc(\eta)+1\} \leq \eta+2
$
and so
$\cind(\xi) \leq \sup\set{\eta+2}{1 \leq \eta < \xi} + 1 = (\eta_0+2)+1 = \xi+2$.
\end{proof}

\begin{question}\normalfont
\label{question estimation of upper bounds}
The preceding upper bounds are not the best possible in general. For example $\bdbc(n) \leq 2$ for all $\n$ and $\bdbc(\om) \leq 3$. It would be interesting to give a precise estimation for $\bdbc(\xi)$ and $\cind(\xi)$.
\end{question}

The analogous result about the Baire property at every $\bolds^0_\xi$-level is given with the help of the preceding functions. Recall the sets $\bcode_\xi^\ca{X}$ given in the Introduction, which parametrize $\bolds^0_\xi \upharpoonright \ca{X}$.

\begin{theorem}
\label{theorem transition from Borel to open codes refined}
For every Polish space \ca{X} and every $1 \leq \xi < \om_1$ there exists a $\bolds^0_{\cind(\xi)}$-measurable function
\[
c^\ca{X}_\xi: \bcodefam_\xi \to \baire
\]
such that the set
\[
\bcode_\xi^\ca{X}(\alpha) \triangle \ocode^\ca{X}(c^\ca{X}_\xi(\alpha))
\]
is meager for all $\alpha \in \bcodefam_\xi$.\smallskip

Moreover one can choose the family $(c^\ca{X}_{\xi})_\xi$ in such a way that $c^{\ca{X}}_\xi \upharpoonright \bcodefam_\eta = c^\ca{X}_\eta$ for all $2 \leq \eta < \xi$.
\end{theorem}

We prove first the following auxiliary lemma. 

\begin{lemma}
\label{lemma from closed to open}
For every Polish space \ca{X}, there is a $\bolds^0_2$-measurable function\newline$\cfo^\ca{X}: \baire \to \baire$ such that the set
\[
\fcode^{\ca{X}}(\alpha) \triangle \ocode^{\ca{X}}(\cfo^\ca{X}(\alpha))
\]
is meager for all $\alpha \in \baire$.\smallskip

{\normalfont (Here ``$\cfo$" stands for ``approximation to closed".)}
\end{lemma}

\begin{proof}
Let us fix for the discussion a closed set $F \subseteq \ca{X}$. It is not hard to verify that the set $F \setminus F^\circ$ is meager, where $F^\circ$ is the interior of $F$. So, given a code for $F$, we need to find an open code for $F^\circ$ in a $\bolds^0_2$-way.  Clearly we have that
\[
x \in F^\circ \iff (\exists s \in \om)[x \in N(\ca{X},s) \ \& \ N(\ca{X},s) \subseteq F].
\]
We now consider the relation $P \subseteq \baire \times \ca{X}$ defined by
\[
P(\alpha,x) \iff (\exists s \in \om)[x \in N(\ca{X},s) \ \& \ N(\ca{X},s) \subseteq \ca{X} \setminus \ocode^{\ca{X}}(\alpha)].
\]
It is clear that $P(\alpha)$ is the interior of $\ca{X} \setminus \ocode^{\ca{X}}(\alpha)$, for all $\alpha \in \baire$. For all $s \in \om$ we put 
\begin{align*}
C_s 
:=& \ \set{\alpha \in \baire}{N(\ca{X},s) \subseteq \ca{X} \setminus \ocode^{\ca{X}}(\alpha)}\\
 =& \ \set{\alpha \in \baire}{N(\ca{X},s) \cap \ocode^{\ca{X}}(\alpha) = \emptyset},
\end{align*}
so that
\begin{align}\label{equation lemma from closed to open A}
P(\alpha,x) \iff  (\exists s \in \om)[x \in N(\ca{X},s) \ \& \ \alpha \in C_s].
\end{align}
Then each $C_s$ is a closed set. To see this let $\alpha \not \in C_s$, \ie $N(\ca{X},s) \cap \ocode^{\ca{X}}(\alpha) \neq \emptyset$. Consider $x \in N(\ca{X},s)$ such that $(\alpha,x) \in \ocode^{\ca{X}}$. Since the latter set is open there are open sets $V \subseteq \baire$ and $W \subseteq \ca{X}$ such that $(\alpha,x) \in V \times W \subseteq \ocode^{\ca{X}}$. In particular $x \in N(\ca{X},s) \cap \ocode^{\ca{X}}(\beta) \neq \emptyset$ for all $\beta \in V$. Hence $\alpha \in V \subseteq \baire \setminus C_s$, and so the complement of $C_s$ is open.

Let us denote by $\ca{T}_{\baire}$ the usual topology on the Baire space. From a well-known result of Kuratowski \cite[22.18]{kechris_classical_dst} there exists a Polish topology $\ca{T}_{\infty}$ on $\baire$ such that $\ca{T}_{\baire} \subseteq \ca{T}_{\infty} \subseteq \bolds^0_2 \upharpoonright (\baire, \ca{T}_{\baire})$, and every $C_s$ is $\ca{T}_{\infty}$-clopen. Moreover we may assume that $\ca{T}_{\infty}$ is zero-dimensional (see the proofs of \cite[13.3, 22.18]{kechris_classical_dst}), so that $(\baire,\ca{T}_{\infty})$ is homeomorphic to a closed subset $F$ of $(\baire,\ca{T}_{\baire})$, see \cite[7.8]{kechris_classical_dst}. In the sequel we consider $F$ with the relative topology inherited from $\ca{T}_{\baire}$. Also, unless stated otherwise, we think of $\baire$ with the usual topology $\ca{T}_{\baire}$.

It follows from (\ref{equation lemma from closed to open A}) that $P$ is an open subset of $(\baire,\ca{T}_{\infty}) \times \ca{X}$. Let $h: F \to (\baire,\ca{T}_{\infty})$ be a homeomorphism. Since $h$ is $\ca{T}_{\infty}$-continuous there is an open set $O \subseteq \baire \times \ca{X}$ such that
\[
P(h(\beta),x) \iff O(\beta,x)
\]
for all $\beta \in F$ and all $x$. We choose some $\ep \in \baire$ such that $O = \ocode^{\ca{X}}(\ep)$, hence for all $x \in \ca{X}$ and all $\beta \in F$ we have
\begin{align}\label{equation lemma from closed to open B}
P(h(\beta),x) \iff \ocode^{\baire \times \ca{X}}(\ep,\beta,x) \iff \ocode^{\ca{X}}(S(\ep,\beta),x),
\end{align}
for some continuous function $S: \baire \times \baire \to \baire$. Finally we define
\[
\cfo^\ca{X}: \baire \to \baire: \alpha \mapsto S(\ep,h^{-1}(\alpha)).
\]
Since $h^{-1}: (\baire,\ca{T}_{\infty}) \to F$ is continuous and $\ca{T}_{\infty} \subseteq \bolds^0_2 \upharpoonright (\baire, \ca{T}_{\baire})$ it follows that $h^{-1} : \baire \to \baire$ is $\bolds^0_2$-measurable. Therefore $\cfo^\ca{X}$ is a $\bolds^0_2$-measurable function.

To check the key property of $\cfo^\ca{X}$, given $\alpha \in \baire$ and $x \in \ca{X}$ we have
\begin{align*}
x \in \left(\fcode^{\ca{X}}(\alpha)\right)^\circ 
\iff& \ x \in \left (\ca{X} \setminus \ocode^{\ca{X}}(\alpha)\right)^\circ\\ 
\iff& \ P(\alpha,x)\\
\iff& \ocode^{\ca{X}}(S(\ep,h^{-1}(\alpha)),x)  \hspace*{15mm} \text{(from (\ref{equation lemma from closed to open B}))}\\
\iff& \ocode^{\ca{X}}(\cfo^\ca{X}(\alpha),x).
\end{align*}
Thus
\begin{align*}
x \in \fcode^{\ca{X}}(\alpha) \triangle \ocode^{\ca{X}}(\cfo^\ca{X}(\alpha))
\iff& \ x \in \fcode^{\baire}(\alpha) \triangle \left (\fcode^{\ca{X}}(\alpha)\right)^\circ\\
\iff& \ x \in \fcode^{\baire}(\alpha) \setminus \left (\fcode^{\ca{X}}(\alpha)\right)^\circ.
\end{align*}
Since the latter set is meager we are done.
\end{proof}

Now we \emph{prove Theorem \ref{theorem transition from Borel to open codes refined}}. In order to define $c^\ca{X}_\xi$ we need some auxiliary functions. 

First we consider the set $D \subseteq \baire$,
\[
D(\gamma) \iff (\forall n)[\rfn{\gamma}(n) \downarrow].
\]
From Remark \ref{remark the set Def is pi02} it follows that the set $D$ is $\boldp^0_2$.

Next we show that there is a recursive (and therefore continuous) partial function $\tau: \baire \rightharpoonup \baire$, which is defined exactly on $D$ and such that
\begin{align}\label{equation theorem transition from Borel to open codes refined A}
\cup_{\n} \ocode^{\ca{X}}(\rfn{\gamma}(n),x) \iff \ocode^{\ca{X}}(\tau(\gamma),x)
\end{align}
for all $\gamma \in D$ and all $x \in \ca{X}$. 

First compute
\begin{align*}
x \in \cup_{\n} \ocode^{\ca{X}}(\rfn{\gamma}(n))
\iff& (\exists n)[x \in \ocode^{\ca{X}}(\rfn{\gamma}(n))]\\
\iff& \ (\exists n,m)[x \in N(\ca{X},\rfn{\gamma}(n)(m))]\\
\iff& \ (\exists t)[x \in N(\ca{X}, \rfn{\gamma}((t)_0)((t)_1)))]\\
\iff& \ (\exists t)[x \in N(\ca{X}, \tau(\gamma)(t))]\\
\iff& \hspace*{2mm} x \in \ocode^{\ca{X}}(\tau(\gamma)),
\end{align*}
where 
\[
\tau(\gamma)(t) =  \rfn{\gamma}((t)_0)((t)_1) = eval(\gamma,(t)_0,(t)_1),
\]
for $\gamma \in  D$ and $t \in \om$, and $eval$ is as in Remark \ref{remark the set Def is pi02}. It follows easily from the latter remark that the function $\tau$ is recursive on $D$, \ie there is a semirecursive set $R^\tau \subseteq D \times \om$ such that for all $\gamma \in D$ and all $s \in \om$ we have that
\[
\tau(\gamma) \in N(\baire,s) \iff R^\tau(\gamma,s).
\]
In the sequel we consider the infinite countable product $\baire^\om$ of the Baire space with the product topology. To make the following formulas easier to read we will denote its members by $\sq{\alpha_i}{\iin}$ instead of $(\alpha_i)_{\iin}$. Moreover we consider a recursive function $\pi : \baire^\om \to \baire$ such that for all $\sq{\alpha_i}{i \in \om}$ in ${\baire}^\om$ and all \n \ we have that
\begin{align}\label{equation theorem transition from Borel to open codes refined B}
\rfn{\pi(\sq{\alpha_i}{i \in \om})}(n) \downarrow \ \& \ \rfn{\pi(\sq{\alpha_i}{i \in \om})}(n) = \alpha_n.
\end{align}
This is easy to achieve using Kleene's Recursion Theorem: we define
\[
\varphi: \baire \times \baire^\om \times \om \to \baire : \varphi(\ep,\sq{\alpha_i}{i \in \om},n) = \alpha_n,
\]
and we consider a recursive $\ep_0$ is such that
\[
\varphi(\ep_0,\sq{\alpha_i}{i \in \om},n) = \rfn{\ep_0}(\sq{\alpha_i}{i \in \om},n) = \rfn{S(\ep_0,\sq{\alpha_i}{i \in \om})}(n),
\]
for some recursive function $S$. (Here we are using that ${\baire}^\om$ is  recursively isomorphic to $\baire$.) We then set $\pi(\sq{\alpha_i}{i \in \om}) = S(\ep_0,\sq{\alpha_i}{i \in \om})$.\smallskip

\emph{Comment.} This is perhaps the only place in the article, where the $\rfn{\alpha}(n)$-coding of the Borel sets seems to be more appropriate than the $(\alpha)_n$ one, since we apply the Kleene Recursion Theorem, see Remark \ref{remark equivalent coding of Borel sets}. However one can actually obtain the analogous function $\pi$ for the $(\alpha)_n$-coding easily: choose a recursive function $\pi' : {\baire}^\om \to \baire$ such that $\pi'(\sq{\alpha_i}{i \in \om})(\langle t,n \rangle) = \alpha_n(t)$ for all $n,t$. Then $(\pi'(\sq{\alpha_i}{i \in \om}))_n = \alpha_n$ for all $n$.\smallskip

Finally we employ the $\bolds^0_2$-measurable function $\cfo^\ca{X}: \baire \to \baire$ from Lemma \ref{lemma from closed to open}. For simplicity we suppress the superscript \ca{X} in $\cfo^\ca{X}$ and $c^\ca{X}_\xi$ for the remaining of this proof. 

We define by recursion on $\xi \geq 1$ the function $c_\xi : \bcodefam_\xi \to \baire$ as follows:
\begin{align*}
c_1(\alpha) =& \ (n \mapsto \rfn{\alpha^\ast}(n)(1)), \hspace*{35mm} \alpha \in \bcodefam_1,\\
c_\xi(\alpha) =& \ (\tau \circ \pi)\left(\sq{\cfo(c_{|\rfn{\alpha^\ast}(n)|}(\rfn{\alpha^\ast}(n)))}{\n}\right),  \quad \alpha \in \bcodefam_\xi, \ \  \xi > 1.
\end{align*}
The functions $c_\xi$ are well defined: if $\alpha \in \bcodefam_\xi$ with $\xi \geq 1$ then $\rfn{\alpha^\ast}(n)$ is defined for all $n$. Moreover from (\ref{equation theorem transition from Borel to open codes refined B}) the function $\pi$ takes values in $D$, which is the domain of $\tau$.

We establish the required properties of the family $(c_\xi)_{\xi}$ in a series of claims.\smallskip

\emph{Claim 1.} For all $2 \leq \eta \leq \xi$ we have $c_{\xi} \upharpoonright \bcodefam_\eta = c_\eta$.\smallskip

The proof of Claim 1 is clear since $\bcodefam_\eta \subseteq \bcodefam_\xi$ and $c_\xi$ and $c_\eta$ are defined by the same formula.\smallskip

\emph{Claim 2.} For all $1 \leq \xi < \om_1$ the function $c_\xi$ is $\bolds^0_{\cind(\xi)}$-measurable.\smallskip

\emph{Proof of Claim 2.} The fact the $c_1$ is continuous is a direct consequence of the last assertion of Remark \ref{remark the set Def is pi02}. (One can also see that $c_1(\alpha)(n) = eval(\alpha^\ast,n,1)$  for all $\alpha \in \bcodefam_1$ and all $n$.) Hence the case $\xi = 1$ is settled. 

Consider a countable ordinal $\xi > 1$ and assume that the assertion is true for all $1 \leq \eta < \xi$.

We recall that a function $f: \ca{X} \to \baire^\om: f = (f_n)_{\n}$ is $\bolds^0_\zeta$-measurable exactly when every function $f_n$ is $\bolds^0_\zeta$ measurable. Since $\tau$ and $\pi$ are continuous functions it suffices to show that for all \n \ the function
\[
f_n: \bcodefam_\xi \to \baire: \alpha \mapsto \cfo(c_{|\rfn{\alpha^\ast}(n)|}(\rfn{\alpha^\ast}(n)))
\]
is $\bolds^0_{\cind(\xi)}$-measurable.

In the remaining of the proof we fix some $n \in \om$. We need to find a $\bolds^0_{\cind(\xi)}$ set $R$ such that for all $\alpha \in \bcodefam_\xi$ and all $i,j \in \om$ we have
\[
R(\alpha,i,j) \iff f_n(\alpha)(i) = j.
\]
Consider the set $P^{\cfo} \subseteq \baire \times \om \times \om$ defined by
\[
P^{\cfo}(\beta,i,j) \iff \cfo(\beta)(i) = j.
\]
Since $\cfo$ is a $\bolds^0_2$-measurable function the set $P$ is $\bolds^0_2$ (in fact $\boldd^0_2$). As it is well-known $P^{\cfo}$ takes the following form:
\[
P^{\cfo}(\beta,i,j) \iff (\exists t)(\forall m)Q(\overline{\beta}(m),m,i,j)
\]
for some $Q \subseteq \om^4$, where $\overline{\beta}(m) = \langle \beta(0),\dots,\beta(m-1)\rangle$; see \cite[4A.1]{yiannis_dst}. For all $\alpha \in \bcodefam_\xi$ and all $i,j$ we have
\begin{align*}
&\ \cfo(c_{|\rfn{\alpha^\ast}(n)|}(\rfn{\alpha^\ast}(n)))(i) = j \iff\\
\iff& \ P^{\cfo}(\cfo(c_{|\rfn{\alpha^\ast}(n)|}(\rfn{\alpha^\ast}(n))),i,j)\\
\iff& \ (\exists t)(\forall m)(\forall u \in \om^{<\om})\\
& \hspace*{1.5mm} [ u = c_{|\rfn{\alpha^\ast}(n)|}(\rfn{\alpha^\ast}(n)) \upharpoonright m \ \longrightarrow \ Q(\langle u(0),\dots,u(m-1)\rangle,m,i,j)]\\
\iff& \ (\exists t)(\forall m)(\forall u \in \om^{<\om})(\forall \eta < \xi)\\
& \hspace*{1.5mm} \{ [ \rfn{\alpha^\ast}(n) \in \bcodefam_\eta \ \& \ u = c_{\eta}(\rfn{\alpha^\ast}(n)) \upharpoonright m ]\\
& \hspace*{45mm} \ \longrightarrow \ Q(\langle u(0),\dots,u(m-1)\rangle,m,i,j)\},
\end{align*}
where in the last equivalence we used Claim 1. From our inductive hypothesis there is for all $\eta < \xi$ a $\bolds^0_{\cind(\eta)}$ relation $Q^\eta \subseteq \baire \times \om^{< \om} \times \om$ such that
\[
Q^\eta(\beta,u,m) \iff c_{\eta}(\beta) \upharpoonright m = u
\]
for all $\beta \in \bcodefam_\eta$ and all $(u, m) \in \om^{< \om} \times \om$. For all $\eta < \xi$ we define $P^\eta \subseteq \baire \times \om^4 \times \om^{<\om}$ by saying that $P^\eta(\alpha,i,j,t,m,u)$ holds exactly when:
\[
\rfn{\alpha^\ast}(n) \downarrow \ \& \ [ \rfn{\alpha^\ast}(n) \in \bcodefam_\eta \ \& \ Q^\eta(\beta,u,m) ] \ \longrightarrow \ Q(\langle u(0),\dots,u(m-1)\rangle,m,i,j),
\]
so that
\[
\cfo(c_{|\rfn{\alpha^\ast}(n)|}(\rfn{\alpha^\ast}(n)))(i) = j \iff (\exists t)(\forall m)(\forall u \in \om^{<\om})(\forall \eta < \xi)P^\eta(\alpha,i,j,t,m,u)
\]
for all $\alpha \in \bcodefam_\xi$ and all $i,j \in \om$. It is easy to see that $P^\eta$ is a $\boldp^0_{\max\{\bdbc(\eta)+1,\cind(\eta)\}}$ set.

The set $R'$ defined by
\begin{align*}
R'(\alpha,i,j,t)
\iff& \ (\forall m)(\forall u \in \om^{<\om})(\forall \eta < \xi)P^\eta(\alpha,i,j,t,m,u)
\end{align*}
is the  countable intersection of sets in $\boldp^0_{\zeta}$, where 
\[
\zeta = \sup\set{\max\{\bdbc(\eta)+1, \cind(\eta)\}}{1 \leq \eta < \xi},
\]
and so it is a $\boldp^0_\zeta$ set as well. Finally we put
\[
R(\alpha,i,j) \iff (\exists t)R'(\alpha,i,j,t)
\]
and we have that $R \in \bolds^0_{\zeta+1} = \bolds^0_{\cind(\xi)}$. Moreover from the preceding equivalences it follows
\begin{align*}
\cfo(c_{|\rfn{\alpha^\ast}(n)|}(\rfn{\alpha^\ast}(n)))(i) = j
\iff& \ (\exists t)(\forall m)(\forall u \in \om^{<\om})(\forall \eta < \xi)P^\eta(\alpha,i,j,t,m,u)\\
\iff& \ (\exists t) R'(\alpha,i,j,t)\\
\iff& \ R(\alpha,i,j)
\end{align*}
for all $\alpha \in \bcodefam_\xi$ and all $i,j \in \om$. \smallskip

\emph{Claim 3.} For all $1 \leq \xi < \om_1$ and all $\alpha \in \bcodefam_\xi$ the set $\bcode_\xi^\ca{X}(\alpha)\triangle \ocode^\ca{X}(c_\xi(\alpha))$ is meager.\smallskip

\emph{Proof of Claim 3.} This is proved by induction on $\xi$. For $\xi=1$ we have for all $\alpha \in \bcodefam_1$ that
\begin{align*}
x \in \bcode^\ca{X}_1(\alpha) 
\iff& \ (\exists n)[x \in \ca{X} \setminus \pi^\ca{X}_0(\rfn{\alpha^\ast}(n))]\\
\iff& \ (\exists n)[x \in N(\ca{X}, \rfn{\alpha^\ast}(n)(1))]\\
\iff& \ (\exists n)[x \in N(\ca{X}, c_1(\alpha)(n))]\\
\iff& \ x \in \ocode^\ca{X}(c_1(\alpha)),
\end{align*}
hence $\bcode^\ca{X}_1(\alpha) = \ocode^\ca{X}(c_1(\alpha))$.

Assume that the assertion is true for all $1 \leq \eta < \xi$. Let $\alpha \in \bcodefam_\xi$, then we have that $\bcode^\ca{X}_\xi(\alpha) = \cup_n \left(\ca{X} \setminus \bcode^\ca{X}_{|\rfn{\alpha^\ast}(n)|}(\rfn{\alpha^\ast}(n))\right)$.

To make the following arguments easier to read we relax the notation further and we put
\begin{align*}
A^c =& \ \ca{X} \setminus A, \hspace*{6.5mm} \text{where} \ A \subseteq \ca{X};\\
\alpha_n =& \ \rfn{\alpha^\ast}(n), \quad \text{for all $n$, where $\alpha \in \bcodefam_\xi$ is fixed from above};\\
\beta_n =& \ c_{|\rfn{\alpha^\ast}(n)|}(\rfn{\alpha^\ast}(n)) = c_{|\alpha_n|}(\alpha_n); \quad \text{for all $n$}.
\end{align*}
Hence $\bcode^\ca{X}_\xi(\alpha) = \cup_n \bcode^\ca{X}(c_{|\alpha_n|}(\alpha_n))^c$.

 From the inductive hypothesis the set
\begin{align}
\bcode^\ca{X}_{|\alpha_n|}(\alpha_n) \ \triangle \ \ocode^\ca{X}(c_{|\alpha_n|}(\alpha_n))
\label{equation theorem transition from Borel to open codes refined C}
=& \ \bcode^\ca{X}_{|\alpha_n|}(\alpha_n)^c \ \triangle \ \ocode^\ca{X}(\beta_n)^c
\end{align}
is meager for all \n. Moreover from the key property of the function $\cfo$ the set
\begin{align}
\fcode^\ca{X}(\beta_n) \ \triangle \ \ocode^\ca{X}(\cfo(\beta_n))
\label{equation theorem transition from Borel to open codes refined D}
=& \ \ocode^\ca{X}(\beta_n)^c \ \triangle \ \ocode^\ca{X}(\cfo(\beta_n)) 
\end{align}
is also meager from all \n. Using that the symmetric differences in (\ref{equation theorem transition from Borel to open codes refined C}) and (\ref{equation theorem transition from Borel to open codes refined D}) are meager we have that the set
\[
\bcode^\ca{X}_{|\alpha_n|}(\alpha_n)^c \ \triangle \ \ocode^\ca{X}(\cfo(\beta_n))
\]
is meager as well for all $\n$. Since $(\cup_n A_n) \triangle (\cup_n B_n) \subseteq \cup_n (A_n \cup B_n)$ we conclude that the set 
\begin{align*}
& \ \ \ \left( \cup_n \bcode^\ca{X}_{|\alpha_n|}(\alpha_n)^c \right)  \ \triangle \ \left(\cup_n \ocode^\ca{X}(\cfo(\beta_n))\right)\\
=& \ \ \ \bcode^\ca{X}_\xi(\alpha) \ \triangle \ \left(\cup_n \ocode^\ca{X}(\cfo(\beta_n))\right)
\end{align*}
is meager. 

It remains to verify that $c_\xi(\alpha)$ is an open code of the set $\cup_n \ocode^\ca{X}(\cfo(\beta_n))$. From the key property (\ref{equation theorem transition from Borel to open codes refined B}) of $\pi$ we have that
\[
\cfo(\beta_n) = \rfn{\pi(\sq{\cfo(\beta_i)}{\iin})}(n)
\]
for all \n, so from the property (\ref{equation theorem transition from Borel to open codes refined A}) of $\tau$ it follows
\begin{align*}
\cup_n \ocode^\ca{X}(\cfo(\beta_n)) 
=& \ \cup_n \ocode^\ca{X}(\rfn{\pi(\sq{\cfo(\beta_i)}{\iin})}(n))\\
=& \ \ocode^\ca{X}(\tau(\pi(\sq{\cfo(\beta_i))}{i \in \om})))\\
=& \ \ocode^\ca{X}\left(\tau(\pi(\sq{\cfo(c_{|\rfn{\alpha^\ast}(i)|}(\rfn{\alpha^\ast}(i)))}{i \in \om}))\right)\\
=& \ \ocode^\ca{X}(c_{\xi}(\alpha)).
\end{align*}
This finishes the proof of Theorem \ref{theorem transition from Borel to open codes refined}.

\subsection*{Optimality of Proposition \ref{proposition continuous transition analytic}}

Our preceding results about the intermediate steps of the Borel hierarchy offer an indication that Proposition \ref{proposition continuous transition analytic} is the best possible, \ie the uniformity function cannot witness the Baire property \emph{for all} $\alpha$ and be continuous (or even Borel) at the same time. Here we prove the latter statement, actually for all levels of the projective hierarchy.

\begin{theorem}
\label{theorem optimal main}
For every $n \geq 1$ and every $\boldd^1_n$-measurable function
\[
f: \baire \to \baire
\]
there is some $\alpha \in \baire$ such that the set
\[
\scode_n^\baire(\alpha) \triangle \ocode^\baire(f(\alpha))
\]
is non-meager.
\end{theorem}

\emph{Remarks.}  The case $n=1$ (analytic sets) in the preceding result is a consequence of some known facts in recursion theory. More specifically, using generic reals, it is possible for every recursive ordinal $\xi$ to find a hyperarithmetical set $P$ such that no open $U$ with $P \triangle U$ being meager has a code recursive in the $\xi$-th Turing jump $\emptyset^{(\xi)}$ of the empty set. The relativized version of the latter implies that there is no Borel-measurable function which takes Borel codes to open codes that witness the Baire property. The assertion of the preceding theorem for $n=1$ follows from the fact that we can transform a Borel code to an analytic code in a recursive way \cite[7B.5]{yiannis_dst}. 

Nevertheless our argument for proving Theorem \ref{theorem optimal main} applies just as good at any level of the projective hierarchy, and there is no need to approach the required level of complexity using intermediate steps like the Turing jumps in the case of analytic sets. Moreover we will actually give two proofs. One in the context of effective descriptive set theory, and a classical one, making it thus accessible to a wider range of researchers.

We also point out that the cases $n > 1$ are proven in the Zermelo-Fraenkel set theory {\bf ZF}. This refutes the existence of a witnessing $\boldd^1_n$-measurable function in \emph{any model} of {\bf ZF} in which $\bolds^1_n$ sets have the Baire property (and not just in the models of ${\bf ZF} + \bolds^1_n$-determinacy.)

Finally the statement of Theorem \ref{theorem optimal main} is also true for $\boldp^1_n$ as well, \ie there does not exist a $\boldd^1_n$-measurable function $f: \baire \to \baire$ such that the set $\neg \scode_n^\baire(\alpha) \triangle \ocode^\baire(f(\alpha))$
is meager for all $\alpha \in \baire$. This is a straightforward application of Lemma \ref{lemma from closed to open}.

The idea in our proof is to construct a set which is simple topologically, and in particular is an open set, but has a ``complex" $\bolds^1_n$-code.

\subsection*{The effective proof}

Fix $k \geq 1$ and suppose towards a contradiction that there were a $\boldd^1_k$-measurable function $f: \baire \to \baire$
such that the set $\scode_k^\baire(\alpha) \triangle \ocode^\baire(f(\alpha))$ is meager for all $\alpha \in \baire$.

We consider an $\ep \in \baire$ such that $f$ is $\Delta^1_k(\ep)$-recursive and some set $P \subseteq \om$ in $\Sigma^1_k(\ep) \setminus \Pi^1_k(\ep)$. We define the set $V \subseteq \baire$ by
\[
V(x) \iff x(0) \in P.
\]
Clearly $V$ is open and belongs to $\Sigma^1_k(\ep)$. Hence $V = \scode^ \baire_k(\alpha)$ for some $\ep$-recursive $\alpha$ and from our assumption towards contradiction the set $V \triangle \ocode(f(\alpha))$ is meager.

It is not hard to see that there exists a recursive function $\tilde{s}: \om \to \om$ such that
\[
N(\baire,\tilde{s}(n)) = \set{x \in \baire}{x(0)=n}
\]
for all \n.  We claim that
\begin{align}
\label{equation counterexample effective A}n \in P \iff (\exists m)[ N(\baire,\tilde{s}(n)) \cap N(\baire,f(\alpha)(m)) \neq \emptyset]
\end{align}
for all \n. The right-hand-side of the preceding equivalence defines a $\Delta^1_k(\ep)$ subset of \om, since
\[
N(\baire,\tilde{s}(n)) \cap N(\baire,f(\alpha)(m)) \neq \emptyset \iff (\exists i)[ r_i \in N(\baire,\tilde{s}(n)) \cap N(\baire,f(\alpha)(m))]
\]
for all \n, where $(r^\baire_i)_{\iin}$ is the fixed recursive presentation of \baire. Hence if we prove (\ref{equation counterexample effective A}) we get a contradiction because $P$ is not a $\Delta^1_n(\ep)$ set.

We now prove equivalence (\ref{equation counterexample effective A}). Suppose that $n$ is in $P$ so that $N(\baire,\tilde{s}(n)) \subseteq V$. If it were $N(\baire,\tilde{s}(n)) \cap N(\baire,f(\alpha)(m)) = \emptyset$ for all $m$, then we would have $N(\baire,\tilde{s}(n)) \cap \ocode(f(\alpha)) = \emptyset$ as well. Hence
\[
N(\baire,\tilde{s}(n)) = N(\baire,\tilde{s}(n)) \setminus \ocode(f(\alpha)) \subseteq V \setminus \ocode(f(\alpha)) \subseteq V \triangle \ocode(f(\alpha)),
\]
which implies that the non-empty set $N(\baire,\tilde{s}(n))$ is meager. Hence there exists some $m$ such that $N(\baire,\tilde{s}(n))\cap N(\baire,f(\alpha)(m)) \neq \emptyset$ and we proved the left-to-right-hand direction. Conversely if $n$ were not a member of $P$ then we would have that $N(\baire,\tilde{s}(n)) \cap V = \emptyset$. Using our hypothesis there is some $m$ such that
\[
\emptyset \neq N(\baire,\tilde{s}(n)) \cap N(\baire,f(\alpha)(m)) \subseteq (\baire \setminus V) \cap \ocode(f(\alpha)) \subseteq V \triangle \ocode(f(\alpha)),
\]
\ie the set $V \triangle \ocode(f(\alpha))$ would contain a non-empty open set, and thus it would not be meager. Hence $n$ is a member of $P$ and we have proved the equivalence (\ref{equation counterexample effective A}). This completes the effective proof.

\subsection*{The classical proof} The idea here is to trace back the contradiction of the preceding proof and bypass any reference to the complexity of subsets of naturals. In order to do this we need a fixed point type theorem, which unsurprisingly comes from effective descriptive set theory.

\begin{theorem}[Kleene's Fixed point theorem for $\bolds^1_n$ relations, see 3H.3 \cite{yiannis_dst}]
\label{theorem fixed point}
Suppose that $\ca{X}$ is a Polish space and that $n \geq 1$. For all $A \subseteq \baire \times \ca{X}$ in $\bolds^1_n$ there exists some $\ep_0 \in \baire$ such that
\[
A(\ep_0,x) \iff \scode_n^{\ca{X}}(\ep_0,x)
\]
for all $x \in \ca{X}$.
\end{theorem}
We present the proof as it appears in \cite{yiannis_dst}.
\begin{proof}
Suppose that $S: \baire \times \baire \to \baire$ be the function which witnesses that our system $(\scode^{\ca{Y}}_n)_{\ca{Y}}$ is good at the instance $\ca{X} = \baire$. We define $B \subseteq \baire \times \ca{X}$ by
\[
B(\alpha,x) \iff A(S(\alpha,\alpha),x).
\]
There exists some $\ep_1 \in \baire$ such that $B = \scode^{\baire \times \ca{X}}_n(\ep_1)$. It follows that
\begin{align*}
A(S(\alpha,\alpha),x) \iff& B(\alpha,x)\\
                                 \iff& \scode_n^{\baire \times \ca{X}}(\ep_1,\alpha,x)\\
                                 \iff& \scode_n^{\ca{X}}(S(\ep_1,\alpha),x)
\end{align*}
for all $x, \alpha$. So if we take $\alpha = \ep_1$ and $\ep_0 = S(\ep_1,\ep_1)$ we are done.
\end{proof}

We now proceed to the classical proof of Theorem \ref{theorem optimal main}. Suppose towards a contradiction that there exists a $\boldd^1_k$-measurable function $f: \baire \to \baire$  such that the set $\scode_k^\baire(\alpha) \triangle \ocode^\baire(f(\alpha))$ is meager for all $\alpha \in \baire$. We define
\[
M(n) = \set{x \in \baire}{x(0)=n}
\]
for all $n$ and $Q \subseteq \baire \times \om$ by
\begin{align*}
Q(\alpha,n) \iff& (\exists m)[M(n) \cap N(f(\alpha)(m)) \neq \emptyset].
\end{align*}
Since both $M(n)$ and $N(f(\alpha)(m))$ are open we have that
\[
Q(\alpha,n) \iff (\exists m,i)[r^\baire_i \in M(n) \cap N(f(\alpha)(m))],
\]
hence $Q$ is a $\boldd^1_k$ set. In particular the complement of $Q$ is $\bolds^ 1_k$ and therefore there exists some $\ep_0$ such that
\begin{align}
\label{equation counterexample classic D}\neg Q(\alpha,n) \iff \scode^\ca{\baire \times \om}_k(\ep_0,\alpha,n)
\end{align}
for all $\alpha, n$. We now define $A \subseteq \baire \times \baire$ as follows
\[
A(\alpha,x) \iff \scode^{\baire \times \om}_k(\ep_0,\alpha,x(0)).
\]
Since $A$ is a $\bolds^1_k$ subset of \ca{X} from Theorem \ref{theorem fixed point} there exists some $\alpha_0$ such that
\[
A(\alpha_0,x) \iff \scode^{\baire}_k(\alpha_0,x)
\]
for all $x$. We now define $V \subseteq \baire$ by
\[
V(x) \iff \scode^{\baire}_k(\alpha_0,x).
\]
By definition $V$ is $\bolds^1_k$ with code $\alpha_0$. Clearly
\begin{align}
\label{equation counterexample classic C}V(x) \iff A(\alpha_0,x) \iff \scode^{\baire \times \om}_k(\ep_0,\alpha_0,x(0)),
\end{align}
and so membership in $V$ depends only on the first coordinate; hence $V$ is open. Moreover we have that
\begin{align}
\label{equation counterexample classic B}M(n) \cap V \neq \emptyset \iff M(n) \subseteq V
\end{align}
for all $n$. From our assumption about $f$ it follows that the set
\[
\scode^\baire_k(\alpha_0) \triangle \ocode(f(\alpha_0)) = V \triangle \ocode(f(\alpha_0))
\]
is meager.

We now claim that
\begin{align}
\label{equation counterexample classic A} V(x) \iff (\exists m)[M(x(0)) \cap N(f(\alpha_0)(m)) \neq \emptyset]
\end{align}
for all $x \in \baire$. To see the latter suppose first that $x$ is in $V$, so that using (\ref{equation counterexample classic B}) $M(x(0)) \subseteq V$. If it where $M(x(0)) \cap N(f(\alpha)(m)) = \emptyset$ for all $m$, then we would have $M(x(0)) \cap \ocode(f(\alpha_0)) = \emptyset$. Hence
\[
M(x(0)) = M(x(0)) \setminus \ocode(f(\alpha_0)) \subseteq V \setminus \ocode(f(\alpha_0)) \subseteq V \triangle \ocode(f(\alpha_0)),
\]
which is a contradiction since no $M(n)$ is a meager set. We prove the converse by contraposition. If $x$ is not a member of $V$, then $M(x(0)) \not \subseteq V$ and again from (\ref{equation counterexample classic B}) we have that $M(x(0)) \cap V = \emptyset$. Hence for all $m$,
\[
M(x(0)) \cap N(\baire, f(\alpha_0)(m)) \subseteq \left(\baire \setminus V \right)\cap \ocode(f(\alpha_0)) \subseteq \ V \triangle \ocode(f(\alpha_0))
\]
Since $V \triangle \ocode(f(\alpha_0))$ is meager, it follows that the open set $M(x(0)) \cap N(\baire, f(\alpha_0)(m))$ is empty for all $m$, \ie the right-hand-side of (\ref{equation counterexample classic A}) fails. Hence we have proved the preceding equivalence.

Putting all together we have that
\begin{align*}
Q(\alpha_0,0) \iff& (\exists m)[M(0) \cap N(f(\alpha_0)(m)) \neq \emptyset] \hspace*{10mm}\textrm{(by definition)}\\
                      \iff& (0,0,0,\dots) \in V\hspace*{34.5mm}\textrm{(from (\ref{equation counterexample classic A}))}\\
                      \iff& \scode^{\baire \times \om}_k(\ep_0,\alpha_0, 0)\hspace*{34.5mm}\textrm{(from (\ref{equation counterexample classic C}))}\\
                      \iff& \neg Q(\alpha_0,0)\hspace*{44mm}\textrm{(from (\ref{equation counterexample classic D}))}
\end{align*}
a contradiction.

\end{document}